\documentclass{article}

\usepackage[utf8]{inputenc}
\usepackage{eufrak}
\usepackage{graphicx}
\usepackage{algorithm}
\usepackage[noend]{algpseudocode}
\usepackage{float}
\usepackage{mathtools}
\usepackage{amssymb,amsmath,stackengine,amsthm}
\usepackage{geometry}
\usepackage{array}
\usepackage{xurl}
\usepackage{cleveref}
\usepackage{tabularx}
\usepackage{subcaption}
\usepackage{diagbox}
\usepackage{makecell}

\newcounter{list}
\newcommand{\listitem}{\stepcounter{list}\thelist.}

\newcommand{\B}[1]{\mathcal{B}_{#1}}

\newcommand{\TL}{\mathcal{T}\mathcal{L}}

\theoremstyle{definition}

\newtheorem{theorem}{Theorem}
\newtheorem{definition}{Definition}
\newtheorem{assumption}{Assumption}
\newtheorem{corollary}{Corollary}

\newtheorem{lemma}{Lemma}

\usepackage{authblk}

\newcommand{\BigO}[1]{\mathcal{O}(#1)}
\newcommand{\Omeg}[1]{\Omega(#1)}

\title{Factorizing the Brauer monoid in polynomial time}

\author[1]{Daniele Marchei}
\author[1]{Emanuela Merelli}
\author[2]{Andrew Francis}
\affil[1]{Computer Science, University of Camerino, Camerino, Italy}
\affil[2]{Centre for Research in Mathematics and Data Science, Western Sydney University, Australia}

\begin{document}
\maketitle

\begin{abstract}

    Finding a minimal factorization for a generic semigroup can be done by using the Froidure-Pin Algorithm, which is not feasible for semigroups of large sizes. On the other hand, if we restrict our attention to just a particular semigroup, we could leverage its structure to obtain a much faster algorithm. In particular, $\BigO{N^2}$ algorithms are known for factorizing the Symmetric group $S_N$ and the Temperley-Lieb monoid $\TL_N$, but none for their superset the Brauer monoid $\B{N}$.
    In this paper we hence propose a $\BigO{N^4}$ factorization algorithm for $\B{N}$. At each iteration, the algorithm rewrites the input $X \in \B{N}$ as $X = X' \circ p_i$ such that $\ell(X') = \ell(X) - 1$, where $p_i$ is a factor for $X$ and $\ell$ is a length function that returns the minimal number of factors needed to generate $X$.
\end{abstract}

\section{Introduction} \label{sec:introduction}
The Brauer monoid $\B{N}$ is a diagram algebra introduced by Brauer in 1937~\cite{brauer1937algebras}, and later utilized by Kauffman and Magarshak in 1995 to define a mapping between RNA secondary structures and elements of $\B{N}$ (which they call ``tangles'') \cite{kauffman1995vassiliev}. The correspondence between properties of RNA secondary structures, and properties of the factorizations of the tangles they get mapped to, can then be studied~\cite{marchei2022rna}. But for this, factorization of Brauer tangles is necessary, and this is the major topic of the present paper.

It can be easily proven that if there exists a length function for $\B{N}$ (i.e. a function that returns the minimal amount of factors the input tangle is generated by) computable in polynomial time, then there exists a polynomial time factorization algorithm (see \Cref{cor:naiveAlgoTimeComplexity}). In this paper we will find two length functions computable in quadratic time, one of which is computable in linear time if we assume some precomputation steps are already done. The key insight is that, since the length function for tangles in the symmetric group $S_N$ (which is a submonoid of $\B{N}$) just involves counting the number of crossings they have, it would be interesting to map every tangle in $\B{N}$ with length $l$ to a tangle in $S_N$ with exactly $l$ crossings. In this way, if this mapping could be done in polynomial time, then we would have found our length function.

The paper is divided in the following way. In \Cref{sec:relatedWorks} we begin by recalling some related work. \Cref{sec:preliminaries} is dedicated to giving some preliminaries on the Brauer monoid, setting up the problem, illustrating a naive $\BigO{N^5}$ algorithm and we outline two assumptions we are basing some proofs on. In \Cref{sec:tauMapping} we go into more detail on the idea of defining a mapping $\tau : \B{N} \rightarrow S_N$, and some of its properties. \Cref{sec:passThrough} will focus on just proving one theorem: i.e. that some edges in a tangle ``cannot come back'' if they satisfy a particular property. This will be the base for \Cref{sec:nodePolarity}, in which we finally outline a quadratic time algorithm for $\tau$. In \Cref{sec:length_functions} we will propose the two length functions that will then be used in \Cref{sec:finalAlgo} to illustrate the final factorization algorithm with $\BigO{N^4}$ time complexity.

\section{Related works} \label{sec:relatedWorks}

In a previous work~\cite{marchei2022rna} we proposed a factorization algorithm that uses a set of polynomial-time heuristics for finding a possibly non-minimal factorization, we then refine it by applying the axioms for the Brauer monoid as a Term Rewriting System (TRS). This will eventually ensure minimality, but since the TRS is not confluent the overall time complexity is difficult to calculate and likely to be huge. This is because for the symmetric group, it is known~\cite{stanley1984number} that the maximal number of minimal factorizations is
\begin{equation*}
    \frac{{N \choose 2}!}{1^{N-1}3^{N-2}\cdots(2N-3)^1}
\end{equation*}
and the TRS will have to check all of them before deciding that there are no more reductions possible.

The Froidure-Pin Algorithm can find the minimal factorization for any element in a finite semigroup $S$ by performing $|S| + |A| - |G| - 1$ operations, where $A$ is the set of axioms of $S$, and $G$ the set of generators~\cite{froidure1997algorithms}. For the Brauer monoid the resulting time complexity lower-bound is therefore $\Omeg{|\B{N}|} = \Omeg{(2N-1)!!}$\footnote{$(2N-1)!!$ is the ``odd double factorial'', defined as $(2N-1)!! = 1\cdot3\cdots(2N-3)\cdot (2N-1)$.}.
    
Algorithm 13 of ``Computing with semigroups'' \cite{east2019computing} is capable of finding a factorization, but it is not guaranteed to be minimal. It computes two sets $\mathfrak{R}$ (the $\mathcal{R}$-classes of $\B{N}$) and $(\B{N})\lambda$. 
The size of $\mathfrak{R}$ for $\B{N}$ can be calculated by the following recurrence relation:
\begin{equation*}
\begin{array}{lll}
    a(0) &=& 1 \\
    a(1) &=& 1 \\
    a(N) &=& a(N-1) + (N-1)a(N-2)
\end{array}
\end{equation*}
which is the number of ways to partition a set of size $N$ into subsets of size one or two \cite{dolinka2017motzkin}. This recurrent relation is clearly bounded below by $N!$. 

The authors also say that for regular semigroups (as in the case for $\B{N}$), $|\mathfrak{R}| = |(\B{N})\lambda|$, and that the calculation for $\mathfrak{R}$ is redundant. This does not reduce the time complexity because $(\B{N})\lambda$ still needs to be computed.

Lastly, two submonoids of $\B{N}$ can be factorized in quadratic time. The symmetric group can be factorized by using the \textsc{BubbleSort} algorithm, and the Temperley-Lieb monoid can be factorized by using the algorithm proposed by Ernst et al.~\cite{ernst2016factorization}. We will discuss these two algorithms in \Cref{sec:app_fact}.

\section{Preliminaries} \label{sec:preliminaries}

Given $N \geq 0$, arrange $2N$ nodes in two rows of $N$ nodes each. Nodes in the upper row are labelled with $[N] = \{1,2,\cdots,N\}$ while nodes in the bottom row are labelled with $[N'] = \{1',2',\cdots,N'\}$. A tangle is a set of $N$ edges connecting any two distinct nodes in $[N] \cup [N']$ such that each node is in exactly one edge. We will represent edges as $e = (x,y)$ in a canonical form in which $x < y$ if both $x$ and $y$ are nodes in the same row, while in the case that $e$ connects nodes from in different rows, $x \in [N]$ and $y \in [N']$.

Given two tangles $X$ and $Y$ we define their composition $X \circ Y$ by stacking $X$ on top of $Y$ (matching the bottom row of $X$ with the top row of $Y$) and then tracing the path of each edge (we will ignore internal loops in our setup). The set of all tangles on $2N$ nodes under composition is called the Brauer monoid $\B{N}$~\cite{brauer1937algebras} and the identity tangle is $I_N = (1,1')(2,2') \cdots (i,i') \cdots (N,N')$ (see \Cref{fig:tangle_example}).

\begin{figure}[ht]
    \centering
    \includegraphics{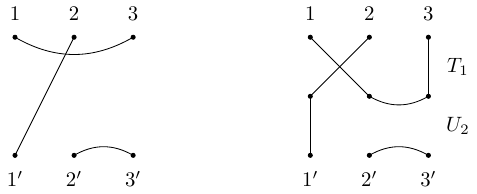}
    \caption{On the left a tangle $X = (1,3)(2,1')(2',3')$ in $\B{3}$, on the right its unique minimal factorization $T_1 \circ U_2$.
    }
    \label{fig:tangle_example}
\end{figure}

For our purposes, it will be useful to classify edges by where they are connected. Tangles have two types of edges~\cite{dolinka2018twisted}:
\begin{itemize}
    \item \textit{hooks} are edges in which both nodes are on the top or on the bottom row. The former ones are called \textit{upper hooks} and the latter \textit{lower hooks};

    \item \textit{transversals} are edges in which one node is on the top row while the other one is on the bottom. We further classify transversal edges as:
    \begin{itemize}
        \item \textit{positive transversal} are in the form $x > y'$
        \item \textit{zero transversal} are in the form $x = y'$
        \item \textit{negative transversal} are in the form $x < y'$
    \end{itemize}
    where $x$ is the upper node and $y'$ is the lower node.
\end{itemize}
Given a tangle $X$, we say that two edges \emph{cross} if they intersect each other in the diagrammatic representation of $X$ (assuming the edges are drawn in a way that minimizes the number of crossings). The \emph{size} of an edge $e = (x, y)$ is defined as $|e| = |x - y|$ ($x$ and $y$ are arbitrary nodes). 

Given a tangle $X$ with an upper hook $h = (i, i+1)$ of size one and another distinct edge $e = (x,y)$, we say that we \emph{merge $h$ with $e$} by removing them from $X$ and connecting their respective nodes such that the newly added edges $e_1$ and $e_2$ do not cross (see \Cref{fig:mergeExamples}). In particular:
\begin{itemize}
    \item if $e$ is a upper hook such that $x < i$ and $i+1 < y$, then $e_1 = (x,i)$ and $e_2 = (i+1,y)$;
    \item if $e$ is a lower hook such that $x \leq i$ and $i+1 \leq y$, then $e_1 = (i,x)$ and $e_2 = (i+1,y)$;
    \item if $e$ is a negative transversal such that $x < i$ and $i+1 \leq y$, then $e_1 = (x,i)$ and $e_2 = (i+1,y)$;
    \item if $e$ is a positive transversal such that $x > i+1$ and $i \geq y$, then $e_1 = (i,y)$ and $e_2 = (i+1,x)$;
    \item for all other combinations, the merging of $h$ and $e$ is undefined, since it will not be useful for our use case (at the end of \Cref{sec:passThrough} it will become clear why).
\end{itemize} 

\begin{figure}[ht]
    \centering
    \newcommand{\scale}{0.85}
    \subfloat[]{\includegraphics[scale = \scale]{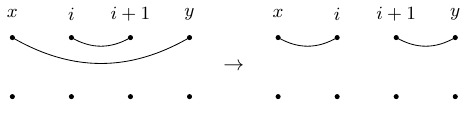}}\hfill
    \subfloat[]{\includegraphics[scale = \scale]{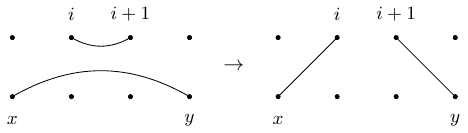}}
    
    \subfloat[]{\includegraphics[scale = \scale]{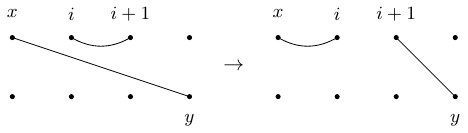}}\hfill
    \subfloat[]{\includegraphics[scale = \scale]{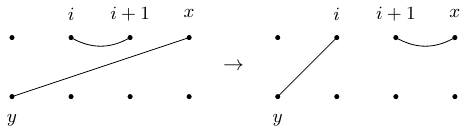}}
    \caption{The four cases for which we define the merge operation. In all cases, $e = (x,y)$ and $h = (i,i+1)$. (a) $e$ is an upper hook. (b) $e$ is a lower hook. (c) $e$ is a negative transversal. (d) $e$ is a positive transversal.}
    \label{fig:mergeExamples}
\end{figure}

The Brauer monoid can also be defined as the set of tangles generated by the composition of \emph{prime} tangles, which we call $T$-primes and $U$-primes, defined as:
\begin{itemize}
    \item $T_i = (1,1')(2,2') \cdots (i,i'+1)(i+1,i') \cdots (N,N')$;
    \item $U_i = (1,1')(2,2') \cdots (i,i+1)(i',i'+1) \cdots (N,N')$.
\end{itemize}
where $i$ is called the index of that prime tangle. These primes also satisfy the axioms in \Cref{tab:axioms}. We classify these axioms into three types:
\begin{itemize}
    \item \textit{Delete rules}: because they decrease the number of primes to be composed;
    \item \textit{Braid rules}: because they resemble the braid relation in the symmetric group;
    \item \textit{Swap rules}: because they allow to commute primes without changing the resulting tangle.
\end{itemize}

We call a word $F \in \{T_i, U_i\}^*$ a \textit{factorization} (the empty word coincides with the identity tangle $I_N$). If this factorization is reduced, i.e. there is no other equivalent word of shorter length, then it is called \textit{minimal}.

Define $\otimes : \B{N} \times \B{M} \rightarrow \B{N+M}$ to be the tensor product such that $Z = X \otimes Y$ is the tangle in which $Y$ is placed on the right of $X$. We call $X$ and $Y$ the components of $Z$~\cite{ram1995characters}. If we assume $F_X$ and $F_Y$ to be minimal factorizations for $X$ and $Y$, then we have that $F_Z = F_X \circ F_Y$ is a minimal factorization for $Z$, therefore we can factorize each component separately and, for the rest of the paper, we will assume that every tangle has only one component.

\begin{table}
    \centering
    \begin{tabular}{llcl}
    Delete &&&\\
        \listitem & $T_i \circ T_i $&=&$ I_N$\\
        \listitem & $U_i \circ U_i $&=&$ U_i$\\
        \listitem & $T_i \circ U_i $&=&$ U_i$\\
        \listitem & $U_i \circ T_i $&=&$ U_i$\\
        \listitem & $U_i \circ U_j \circ U_i$&=&$U_i \iff |i-j| = 1$\\
        \listitem & $U_i \circ T_j \circ U_i$&=&$U_i \iff |i-j| = 1$\\
        \listitem & $T_i \circ U_j \circ U_i$&=&$T_j \circ U_i \iff |i-j| = 1$\\
        \listitem & $U_i \circ U_j \circ T_i$&=&$U_i \circ T_j \iff |i-j| = 1$\\
        \listitem & $U_i \circ T_j \circ T_i$&=&$U_i \circ U_j \iff |i-j| = 1$\\
        \listitem & $T_i \circ T_j \circ U_i$&=&$U_j \circ U_i \iff |i-j| = 1$\\
    Braid &&&\\
        \listitem & $T_i \circ T_j \circ T_i$ &=&$ T_j \circ T_i \circ T_j \iff |i-j| = 1$\\
        \listitem & $T_i \circ U_j \circ T_i$ &=&$ T_j \circ U_i \circ T_j \iff |i-j| = 1$\\
    Swap &&&\\
        \listitem & $T_i \circ T_j$ &=&$ T_j \circ T_i \iff |i-j| > 1$\\
        \listitem & $T_i \circ U_j$ &=&$ U_j \circ T_i \iff |i-j| > 1$\\
        \listitem & $U_i \circ T_j$ &=&$ T_j \circ U_i \iff |i-j| > 1$\\
        \listitem & $U_i \circ U_j$ &=&$ U_j \circ U_i \iff |i-j| > 1$\\
    \end{tabular}
\caption{Axioms for the Brauer monoid. Axioms from 1 to 10 are called \textit{delete rules}, 11 and 12 are \textit{braid rules} and 13-16 are \textit{swap rules}.}
\label{tab:axioms}
\end{table}

A length function $\ell(X) : \B{N} \rightarrow \mathbb{N}$ is a function that returns the minimal number of prime factors required to compose the tangle $X$, in other words, it returns the length of a minimal factorization for $X$ (define $\ell(I_N) = 0$). If $\ell(X)$ can be computed in polynomial time, say $t(N)$, then \Cref{algo:naiveFactorization} returns a minimal factorization in $\BigO{N^{3}t(N)}$ (see \Cref{cor:naiveAlgoTimeComplexity} for the proof).

\begin{algorithm}[ht]
\caption{Finds the minimal factorization for any tangle in $\B{N}$.}
\label{algo:naiveFactorization}
\begin{algorithmic}
\Require {$X \in \B{N}$}
\Function{Factorize $\B{N}$}{$X$}
    \State $l \gets \ell(X)$
    \State $F \gets$ empty list
    \While {$l \neq 0$}
        
    \If{$(i,i+1) \in X$}
        \State $h \gets (i,i+1)$
        \For{$e \in X, e \neq h$}
            \State $X' \gets $ merge $h$ with $e$ in $X$ (if defined)
            \If{$\ell(X') = l - 1$}
                \State $X \gets X'$
                \State Append $U_i$ to $F$
                \State \textbf{break}
            \EndIf
        \EndFor
    \Else
        \For{$i \in [1 \dots N-1]$}
            \State $X' \gets T_i \circ X$
            \If{$\ell(X') = l - 1$}
                \State $X \gets X'$
                \State Append $T_i$ to $F$
                \State \textbf{break}
            \EndIf
        \EndFor
    \EndIf
    \State $l \gets l - 1$
    \EndWhile
    \State \textbf{return} $F$
\EndFunction
\end{algorithmic}
\end{algorithm}

Lastly, we list two assumptions we believe are true but were not able to prove. They will be useful for some proofs in the following Sections.

\begin{assumption} \label{ass:reductionInBN}
Every factorization in the Brauer monoid can be reduced to a minimal one by a sequence of ``delete'', ``braid'' and ``swap'' rules. In other words, we do not need to increase the factorization length in order to find a shorter one.
\end{assumption}

\begin{assumption} \label{ass:T_are_crossings}
    If a tangle $X$ has $k$ crossings, then there exists a minimal factorization $F$ with exactly $k$ $T$-primes and no other factorization with fewer $T$-primes exists. 
\end{assumption}

We empirically tested these assumptions, for the methodology we refer to \Cref{sec:testAssumptions}. 

\section{Mapping $\B{N}$ to $S_N$} \label{sec:tauMapping}

Tangles in the symmetric groups $S_N$ are generated only by $T$-primes. This implies that, given a tangle $X$ in $S_N$, calculating $\ell(X)$ will amount to just counting its number of crossings, which can be done in $\BigO{N^2}$ time\footnote{There is a faster approach that brings down the time complexity to $\BigO{N\log N}$~\cite{kleinberg2006algorithm}, but as we will see in \Cref{sec:nodePolarity}, for our purposes it will be much more useful to have the actual factorization for $X$, which cannot be done faster than $\BigO{N^2}$.}.
With this in mind, it would be useful to have a function $\tau : \B{N} \rightarrow S_N$ that maps tangles $X \in \B{N}$ to tangles in $S_N$ such that $\ell(X) = \ell(\tau(X))$. In this way, we could define $\ell(X)$ by just computing $\tau(X)$ and counting its number of crossings. Therefore if $\tau$ can be computed in polynomial time, then $\ell$ can be computed in polynomial time too. We now proceed to define $\tau$.

\begin{definition}
Given a factorization $F$ for a tangle $X \in \B{N}$, we define $\tau$ as the function that maps each prime factor $p_i$ of $F$ to $T_i$. In other words, $T_i\mapsto T_i$ and $U_i\mapsto T_i$ for each $T$ and $U$ prime in $F$.  
\end{definition}

\begin{theorem} \label{thm:tauMinimal}
    If $F$ is a minimal factorization, then $\tau(F)$ is minimal too.
\end{theorem}

\begin{proof}
    We prove this by contradiction (see \Cref{fig:thmTauMinimal}).

    Let $F$ be a minimal factorization and assume $\tau(F)$ is not minimal. Therefore it can be rewritten to contain the subword $T_i\circ T_i$ by using Axioms 11 and 13 (\cite{bjorner2005combinatorics}, Theorem 3.3.1). 

    Let $s = r_1r_2\dots r_{k-1}$ be a sequence of rewritings for $\tau(F)$ using Axioms 11 or 13 and $r_k$ be a rewriting step using Axiom 1. For each $r_i \in s$, there exists at least one rewriting $\bar{r}_i$ in the preimage $\tau^{-1}(r_i)$, meaning that $r_i$ applied to $\tau(F)$ maps to a rewriting step $\bar{r}_i \in \tau^{-1}(r_i)$, applied to $F$, that uses Axioms 11 to 16.

    This implies that there exists a sequence of rewritings $\bar{s} = \bar{r}_1\bar{r}_2\dots \bar{r}_{k-1}$ for $F$ such that, after $\bar{r}_{k-1}$, $F$ will contain one of the following subwords: $T_i \circ T_i$, $U_i \circ U_i$, $T_i \circ U_i$ or $U_i \circ T_i$. Therefore, $\bar{r}_k$ will use one of the Axioms from 13 to 16 to reduce $F$ to a shorter factorization.

    This implies that $F$ was not minimal, which is a contradiction.

\end{proof}

\begin{figure}
    \centering
    \includegraphics[scale = 1]{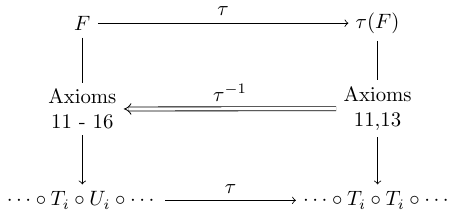}
    \caption{Diagram for the proof of \Cref{thm:tauMinimal}. Assume $F$ is minimal and $\tau(F)$ is not. Axioms 11 and 13 applied to $\tau(F)$ have preimage to axioms 11-16 applied to $F$. In the end, $\tau(F)$ will contain a $T_i \circ T_i$ because we assumed it was not minimal, which implies $F$ was not minimal too and thus we have a contradiction.}
    \label{fig:thmTauMinimal}
\end{figure}

\begin{corollary}
    Let $F$ be an arbitrary factorization. If $\tau(F)$ is not minimal, then $F$ is not minimal too.
\end{corollary}
\begin{proof}
    This is the contrapositive of \Cref{thm:tauMinimal}.
\end{proof}

\begin{corollary} \label{cor:FandTauFSameLength}
    Assuming $F$ to be minimal, then $|F| = |\tau(F)|$.
\end{corollary}

\Cref{cor:FandTauFSameLength} allows us to determine the maximal length a minimal factorization in $\B{N}$ can have.

\begin{corollary} \label{cor:longestMinimalFact}
    The longest minimal factorization in $\B{N}$ has length $\frac{N(N-1)}{2}$.
\end{corollary}
\begin{proof}
    By \Cref{thm:tauMinimal} we know that every factorization $F$ has another factorization $\tau(F)$ with the same length composed only by $T$-primes. Therefore we only need to check the longest minimal factorization in $S_N$, which is well-known to be $\frac{N(N-1)}{2}$.
\end{proof}

\begin{corollary}\label{cor:naiveAlgoTimeComplexity}
    \Cref{algo:naiveFactorization} has time complexity $\BigO{N^{3}t(N)}$, where $t(N)$ is the time complexity for the length function $\ell$.
\end{corollary}
\begin{proof}
    In the case in which there exists edge $h = (i, i+1)$ in $X$, \Cref{algo:naiveFactorization} iterates through all edges, merges them with $h$ and then calculates $\ell$ for each of the resulting tangle. The overall complexity for this case is therefore $\BigO{Nt(N)}$. By \Cref{cor:longestMinimalFact} we know that the longest factorization possible is quadratic and \Cref{algo:naiveFactorization} removes each of them one at a time. Therefore the time complexity for \Cref{algo:naiveFactorization} is $\BigO{N^3t(N)}$. 
\end{proof}

The function $\tau$ we defined can be generalized to operate on an arbitrary subword of a factorization $F$. 

\begin{definition}
Let $F$ be a factorization for a tangle $X \in \B{N}$, we define $\tau^*$ be the function that applies $\tau$ to only a subword of $F$.
\end{definition}

\begin{theorem} \label{thm:partialTau}
    If $F$ is a minimal factorization, then $\tau^*(F)$ is minimal too.
\end{theorem}
\begin{proof}
    The proof is similar to the one for \Cref{thm:tauMinimal} but it requires \Cref{ass:reductionInBN} because $\tau^*(F)$ can be any factorization in $\B{N}$, and every braid/swap $r$ applied to $\tau^*(F)$ has to correspond to a braid/swap in $F$ that is in the preimage of $\tau^*(r)$. There are also more ways in which $\tau^*(F)$ can be not minimal. See the following Hasse diagram:
    \begin{figure}[H]
        \centering
        \includegraphics{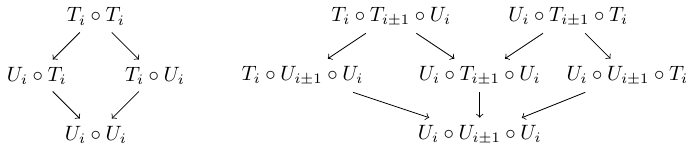}
        \label{fig:hasse_tau_preimage}
    \end{figure}
    The arrows represent the preimage of $\tau^*$, each element has also a self-loop that was not drawn. For the proof of \Cref{thm:tauMinimal} only the left Hasse Diagram was applicable, but now also the right one has to be taken into consideration. Since we assumed that $\tau^*(F)$ was not minimal, then by \Cref{ass:reductionInBN} it can be rewritten by a sequence of swaps and braids to contain any element in the above Hasse Diagram, but for all such elements there exists a preimage that must be in $F$ and, since all preimages are not minimal, it implies that $F$ was not minimal in the first place, hence we have a contradiction and $\tau^*(F)$ must have been minimal too.
\end{proof}

\begin{corollary} \label{cor:tauStarNotMinimalImpliesNotMinimal}
    Let $F$ be an arbitrary factorization. If $\tau^*(F)$ is not minimal, then $F$ is not minimal too.
\end{corollary}
\begin{proof}
    This is the contrapositive of \Cref{thm:partialTau}.
\end{proof}

\Cref{cor:tauStarNotMinimalImpliesNotMinimal} will be very helpful in \Cref{sec:passThrough}, where we will use it to prove that some undesirable factorizations are always not minimal.

We can now use the above theorems to find a (not very useful) length function for $\B{N}$. Given a tangle $X \in \B{N}$, assume to have a minimal factorization $F$ for $X$. Now compute $\tau(F)$, which corresponds to another tangle $ \tau(X) \in S_N$. By \Cref{cor:FandTauFSameLength} we know that the number of crossings of $\tau(X) = |F|$ so we just count them and we obtain the number of factors for $X$. We can represent this mapping as the diagram in \Cref{fig:tauDiagram}.
\begin{figure}[H]
    \centering
    \includegraphics[scale = 1]{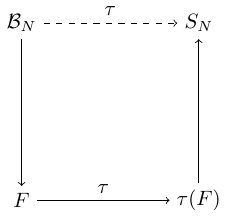}
    \caption{From a tangle in $\B{N}$ we obtain a minimal factorization $F$, we then compute $\tau(F)$ which corresponds to a tangle in $S_N$.}
    \label{fig:tauDiagram}
\end{figure}
Admittedly, this is not a very useful mapping because we are assuming to have $F$ in the first place (which is what we are ultimately looking for). What we would like to find now is a way to extend $\tau$ to arbitrary tangles, not just factorizations (the dashed arrow in \Cref{fig:tauDiagram}). In this way, if we find a polynomial time algorithm for finding a tangle $\tau(X) \in S_N$ such that its number of crossings is equal to $\ell(X)$, we would then have a polynomial time length function. We will present such an algorithm in \Cref{sec:nodePolarity}.

\section{Passing through and coming back} \label{sec:passThrough}

This Section is entirely dedicated to proving that if the number of $T$-primes an edge $e$ ``passes through'' is greater or equal to $|e|$, then it does not pass through any $U$-prime. This is a key property we will leverage in \Cref{sec:nodePolarity}, where at the end of it, we will have all the necessary components for computing $\tau(X)$ in polynomial time. In this Section we will also argue why the merge operation presented in \Cref{sec:preliminaries} is undefined for some edges.

\begin{definition}[Passing through]
    Given a factorization $F$ for a tangle $X$, a prime tangle $p_i \in F$ and an edge $e \in X$, we say that $e$ ``passes through'' $p_i$ if $e$ is connected to the node $i$ or $i'$ of $p_i$. We indicate with $\#T(e)$ and $\#U(e)$ the number of $T$-primes and $U$-primes respectively $e$ passes through in $F$.
\end{definition}

We will now define what ``coming back'' means. The goal for these definitions is to prove that if $\#U(e) > 0$ then $e$ cannot come back. This will imply that $\#U(e) > 0 \implies \#T(e) < |e|$ and therefore, by contraposition, $\#T(e) \geq |e| \implies \#U(e) = 0$, which is what we want to prove.

\begin{definition}[Coming back at $p$]
    Given a factorization $F$ for a tangle $X$, a prime tangle $p \in F$ and an edge $e \in X$, we say that $e$ ``comes back'' at $p$ if $|e|$ decreases when passing through $p$.
\end{definition}

\begin{definition}[Coming back in $F$]
    Given a factorization $F$ for a tangle $X$ and an edge $e \in X$, we say that $e$ ``comes back'' in $F$ if there exists a prime $p$ at which $e$ comes back.
\end{definition}

See \Cref{fig:passes_through_example} for an example.

\begin{figure}
        \centering
        \includegraphics[width = .5\textwidth]{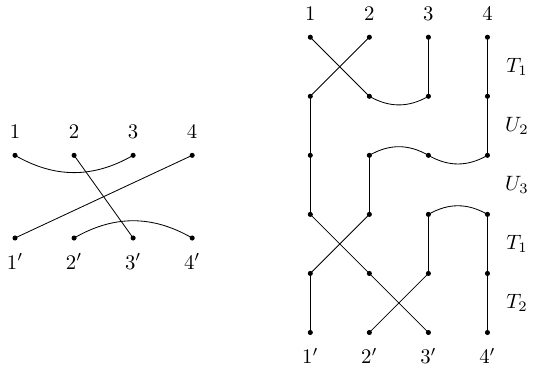}
        \caption{A tangle along with one of its minimal factorizations $ F = T_1 \circ U_2 \circ U_3 \circ T_1 \circ T_2$. The edge $(1,3)$ passes through $T_1$ and $U_2$, edge $(2,3')$ passes through two $T_1$s and one $T_2$, edge $(4,1')$ passes through $U_3, U_2$ and $T_1$ and edge $(2',4')$ passes through $T_2$ and $U_3$. Edge $(2,3')$ comes back at the bottom-most $T_1$ because it decreases in size, this also means that it comes back in $F$. Note also that $(2,3')$ is the only edge that passes through only $T$-primes and satisfies the property $\#T(e) \geq |e|$.}
        \label{fig:passes_through_example}
    \end{figure}

\begin{theorem} \label{thm:comesBackSumPrimes}
    Let $X \in \B{N}$, $e \in X$ and $F$ any minimal factorization for $X$. Then $e$ comes back in $F$ if and only if $\#T(e) + \#U(e) > |e|$.
\end{theorem}
\begin{proof}
    Forward direction.

    Assume $e$ comes back in $F$, then there exists a prime $p_i$ at which $e$ comes back, but if it comes back at $p_i$ it means $e$ passed through another prime $p_i'$ with the same index $i$, and since every edge passes through at least $|e|$ distinct primes, it must be that $\#T(e) + \#U(e) > |e|$.

    The backward direction uses the same argument.

\end{proof}

We will now use \Cref{thm:comesBackSumPrimes} to prove that a tangle in $S_N$ has a pair of edges $e$ and $e_2$ as in \Cref{fig:thm5_and_6}a if and only if $e$ comes back. This theorem will allow us to ignore these tangles in the following theorems, because later on we will assume that a given edge does not come back.

\begin{theorem} \label{thm:existsEdgeThenComesBack}
    Given $X \in S_N$ with a negative transversal $e_1 = (x_1,y_1')$, if there exists another edge $e_2 = (x_2,y_2')$ such that $x_2 < x_1$ and $y_2' > y_1'$ then $e_1$ comes back in any factorization for $X$.
\end{theorem}
\begin{proof}
    Let's set some variables:
    \begin{equation*}    
    \begin{array}{rcl}
        a & = & x_2 - 1 \\
        b & = & x_1 - x_2 - 1 \\
        c & = & N - x_1 \\
        a' & = & y_1' - 1 \\
        b' & = & y_2' - y_1' - 1 \\
        c' & = & N - y_2'
    \end{array}
    \end{equation*}
    where $a, b$ and $c$ are the number of nodes with indices less than $x_2$, in between $x_2$ and $x_1$, and greater than $x_1$ respectively (same for $a', b'$ and $c'$). Therefore we have that 
    \begin{equation*}
        a + b + c = a' + b' + c'
    \end{equation*}
    implying that
    \begin{equation*}
        a' - a - b = c - b' - c'
    \end{equation*}
    which counts the least amount of edges that have to cross $e$ from right to left. Notice now that
    \begin{equation*}
        \begin{array}{ccl}
            x_1 &=& b + x_2 + 1 = b + a + 2 \\
            y_1' & = & a' + 1 
        \end{array}
    \end{equation*}
    and therefore
    \begin{equation*}    
        |e_1| = y_1' - x_1 = a' +1 - (b+a+2) = a' - a - b - 1        
    \end{equation*}
    By substitution, we can now obtain
    \begin{equation*}
        \begin{array}{ccc}
            a' - a - b & = & c - b' - c' \\
            |e_1| + 1 & = & c - b' - c' \\
            |e_1| & < & c - b' - c'
        \end{array}
    \end{equation*}
    which shows that the least amount of edges that have to cross $e_1$ is bigger than $|e_1|$, which implies that $\#T(e_1) > |e_1|$ and therefore that $e_1$ comes back in $F$ (\Cref{thm:comesBackSumPrimes}).

    Note that this proof is not dependent on the factorization chosen, therefore $e_1$ comes back in all minimal factorization for $X$.
\end{proof}

\begin{figure}
    \centering
    \subfloat[]{\includegraphics[width = 0.45\textwidth]{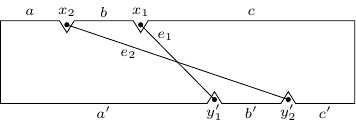}}\hfill
    \subfloat[]{\includegraphics[width = 0.45\textwidth]{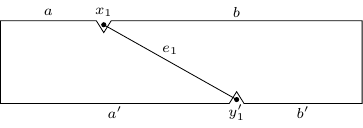}}

    \caption{(a) Illustration for \Cref{thm:existsEdgeThenComesBack}. If a tangle $X$ in $S_N$ contains an edge $e_1$ and an edge $e_2$ as shown, then $e_1$ will come back in all minimal factorizations for $X$. (b) Illustration for \Cref{thm:ComesBackThenExistsOtherEdge}. If $e_1$ comes back in a minimal factorization, then there exists another edge $e_2$ like in (a).}
    \label{fig:thm5_and_6}
\end{figure}

\begin{theorem} \label{thm:ComesBackThenExistsOtherEdge}
    Given $X \in S_N$ with a negative transversal $e_1 = (x_1,y_1')$ and minimal factorization $F$ for $X$, if $e_1$ comes back in $F$, then there exists another edge $e_2 = (x_2,y_2')$ such that $x_2 < x_1$ and $y_2' > y_1'$ (\Cref{fig:thm5_and_6}b).
\end{theorem}
\begin{proof}
    Since $e_1$ comes back in $F$, then $\#T(e_1) > |e_1|$ (\Cref{thm:comesBackSumPrimes}). Therefore $\#T(e_1) = |e_1| + c$ where $c \geq 1$.

    Let's set some variables:
    \begin{equation*}    
    \begin{array}{rcl}
        a & = & x_1 - 1 \\
        a' & = & y_1' - 1 = a + |e_1|
    \end{array}
    \end{equation*}
    
    Where $a$ is the number of nodes before $x_1$ and $a'$ is the number of nodes before $y'$.
    Suppose now that there is no edge $e_2 = (x_2,y_2')$ such that $x_2 < x_1$ and $y_2' > y_1'$, therefore we are assuming that all $|e_1| + c$ edges that cross $e_1$ connect nodes with indices greater than $x_1$ to nodes with indices less than $y_1'$.

    The number of nodes with indices less than $y_1'$ that do not cross $e$ is therefore
    \begin{equation*}
    \begin{array}{rcl}
        a' - (|e_1| + c) & = & \\
        a + |e_1| - |e_1| - c & = & \\
        a - c
    \end{array}
    \end{equation*}

    Since we know that $c \geq 1$, we have that $a - c < a$, which implies that not all nodes with indices less than $x_1$ can connect to lower nodes with indices less than $y_1'$. Therefore there must be another edge $e_2 = (x_2,y_2')$ that crosses $e_1$ with $x_2 < x_1$ and $y_2' > y_1'$.
    
\end{proof}

\begin{corollary} \label{cor:SN_comes_back_iff_other_edge}
    Given $X \in S_N$ with a negative transversal $e = (x_1,y_1')$, then $e$ comes back in every minimal factorization for $X$ if and only if there exists another edge $e_2 = (x_2,y_2')$ such that $x_2 < x_1$ and $y_2' > y_1'$.
\end{corollary}

In light of \Cref{cor:SN_comes_back_iff_other_edge}, from now on we will just say that $e$ does not come back in $X$, since it is independent of the factorization.

Assuming that we now have a tangle $X \in S_N$ with an edge $e$ that does not come back, we can always decompose it as $X = L \circ Z \circ R$ (see \Cref{fig:LZRDec_ex}). This decomposition will come in handy when we will have to prove that a particular factorization is always not minimal, and it will turn out that the two tangles $L$ and $R$ can be safely ignored, simplifying the proof.

\begin{theorem}[LZR Decomposition] \label{thm:LZRDecompositionTheorem}
    Let $a,b \geq 0$ be natural integers. Let $a', b' \geq 0 $ be natural integers such that $a' > a$ and $a + b = a' + b'$.  Let $X \in S_{a+b+1}$ be a tangle containing the edge $e = (a+1,a'+1)$ that does not come back. Then $X$ can be decomposed as a composition of three tangles $L, Z, R \in S_{a+b+1}$, where $L = I_{a+1} \otimes \sigma$ with $\sigma \in S_b$, $Z = T_{a+1} \circ T_{a+2} \circ \cdots T_{a+|e|}$ (containing the edge $e$) and $R = \omega \otimes I_{b' + 1}$ with $\omega \in S_{a'}$. This decomposition satisfies the property $\ell(X) = \ell(L) + \ell(Z) + \ell(R)$. In the special case in which $a'+1 = a+b+1$ then $L = I_N$.
\end{theorem}
\begin{proof}
    Suppose a minimal factorization $F$ for $X$ is given. Since we know that $e$ does not come back, then we also know that $\#T(e) = |e|$. This implies that $F$ contains the sub-word $Z = T_{a+1} \circ T_{a+2} \circ \cdots T_{a+|e|}$. Now use the axioms to rewrite $F$ so that $F = L' \circ Z \circ R'$, where $L'$ and $R'$ contain the rest of the factorization. Finally, use Axiom 13 to move every prime $p_i$ in $L'$ with $i < a$ into $R'$ and move every prime $p_i$ in $R'$ with $i > a'+1$ into $L'$. Thus obtaining $F = L \circ Z \circ R$. This obviously satisfies $\ell(X) = \ell(L) + \ell(Z) + \ell(R)$.

    In the case in which $N = y' = a+b+1$, we can construct $R$ by removing edge $e$ from $X$, relabelling all upper nodes $d$ such that $a+2 \leq d \leq N$ as $d-1$ and adding the edge $(N,N')$. By construction, we have then that $X = Z \circ R$ and therefore we can set $L = I_N$.
\end{proof}

\begin{figure}
    \centering
\subfloat[]{\includegraphics{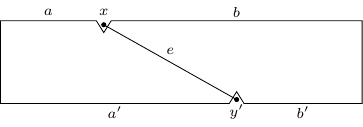}\vspace{0.65cm}}\hfill
\subfloat[]{\includegraphics{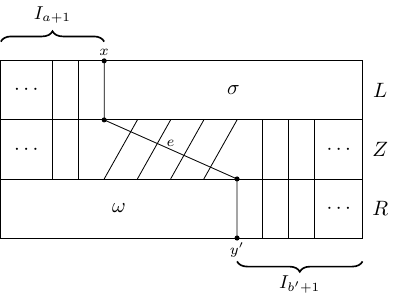}}
 
    \caption{(a) A tangle $X \in S_N$ with an edge $e = (x,y')$ that does not come back. (b) The LZR decomposition of $X$. $\sigma$ and $\omega$ are tangles in $S_b$ and $S_{a'}$, while $Z = T_{a+1} \circ T_{a+2} \circ \cdots \circ T_{a+|e|}$.}
    \label{fig:LZRDec_ex}
\end{figure}

Finally, we can now prove that if an edge $e$ passes through at least one $U$-prime then it cannot come back (see \Cref{fig:goingBackExample}). We assume $U_i$ to be the last $U$-prime edge $e$ passes through before passing through prime tangle $p$. This is because in the case in which $e$ passed through another $U$-prime, say $U_j$, in a different factorization, we could pick $U_j$ as the focus of the proof. This implies that, after $U_i$, $e$ passes through only $T$-primes, in a generic tangle $A$, before $p$. We will also assume that it is the first time that $e$ comes back. 
This assumption also allows us to say that, since $e$ cannot come back the first time, it can never come back.

Lastly, we prove only the case in which $e$ is a negative transversal in $A$ and $U_i$ is on top of $A$. This is because we can reflect $U_i \circ A \circ T_j$ vertically or horizontally to cover any case, and since reflecting a tangle does not change the length of its factorization, then if one of them is not minimal, then all reflections are not minimal too. This implies that an edge cannot come back if at \textit{any time} it passes through a $U$-prime.
\begin{figure}
    \centering
    \includegraphics[width = 0.4\textwidth]{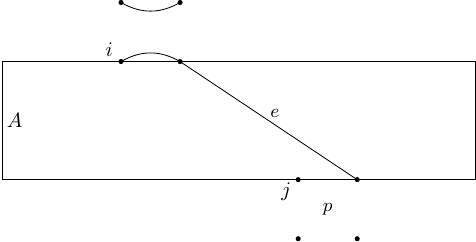}
    \caption{The general case in which an edge $e$ ``comes back''. $e$ passes through a $U$-prime $U_i$, then, after a series of $T$-primes contained in a tangle $A \in \B{N}$, it passes through a prime tangle $p_j$ that reduces it size. We would like to prove that the factorization $U_i \circ F_A \circ p_j$ (where $F_A$ is a minimal factorization for $A$) is always not minimal.}
    \label{fig:goingBackExample}
\end{figure}

We now prove that the configuration seen in \Cref{fig:goingBackExample} always results in a not minimal factorization. By using \Cref{thm:partialTau} and \Cref{thm:LZRDecompositionTheorem} we can now reduce the general case of \Cref{fig:goingBackExample} to a simpler one, in which we prove the non-minimality of the factorization $U_i \circ Z \circ Z' \circ T_{j}$, where $Z$ and $Z'$ are the results of two LZR decompositions and thus have the form $Z = T_{i+1} \circ T_{i+2} \circ \cdots \circ T_{j}$ and $Z' = T_{k} \circ T_{k+1} \circ \cdots \circ T_{j-1}$ for some $k < j$. Follow \Cref{fig:simplerCase} for the steps of the proof.

Once we are in the case of \Cref{fig:simplerCase}i we will have a factorization in this form:
\begin{equation*}
    U_i \circ \underbrace{T_{i+1} \circ T_{i+2} \circ \cdots \circ T_{j}}_{Z} \circ \underbrace{T_{k} \circ T_{k+1} \circ \cdots \circ T_{j-1}}_{Z'} \circ T_{j}
\end{equation*}

\begin{figure}
    \centering
    \newcommand{\scale}{0.33}
    \subfloat[]{\includegraphics[width = \scale\textwidth]{figs/comes_back/A.pdf}}\hfill
    \subfloat[]{\includegraphics[width = \scale\textwidth]{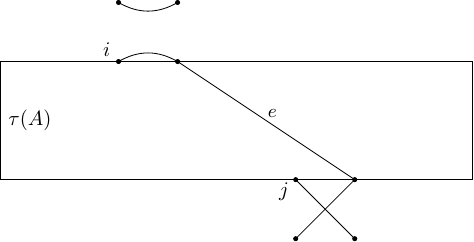}}\hfill
    \subfloat[]{\includegraphics[width = \scale\textwidth]{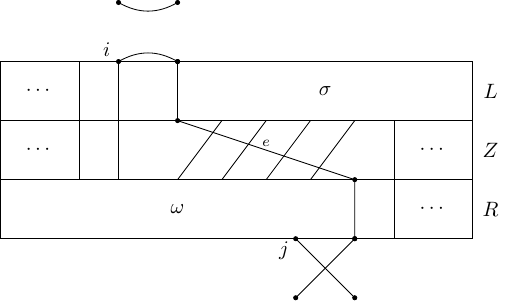}}\hfill\\[3ex]
    \subfloat[]{\includegraphics[width = \scale\textwidth]{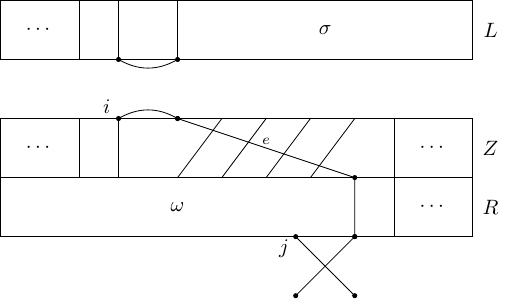}}\hfill
    \subfloat[]{\includegraphics[width = \scale\textwidth]{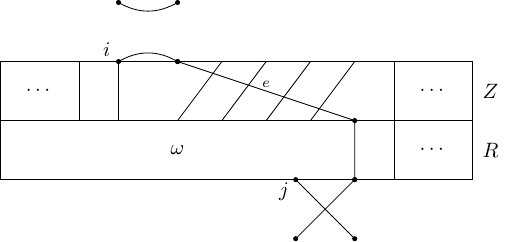}}\hfill
    \subfloat[]{\includegraphics[width = \scale\textwidth]{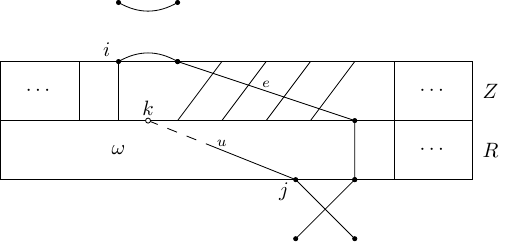}}\hfill\\[3ex]
    \subfloat[]{\includegraphics[width = \scale\textwidth]{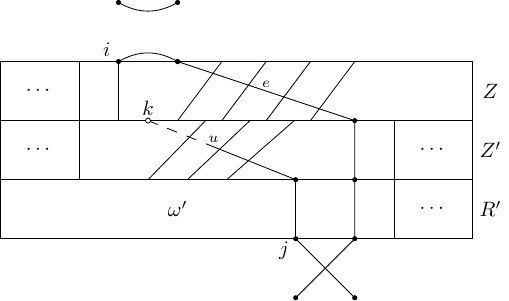}}\hfill
    \subfloat[]{\includegraphics[width = \scale\textwidth]{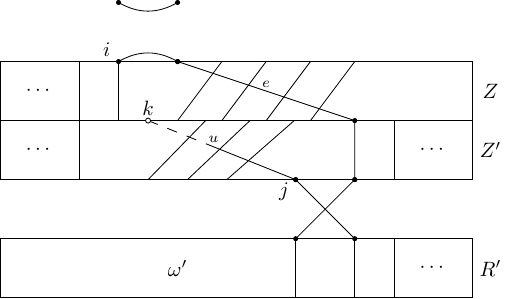}}\hfill
    \subfloat[]{\includegraphics[width = \scale\textwidth]{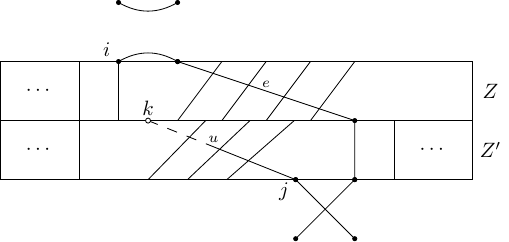}}\hfill\\[3ex]

    \caption{(a) The original tangle $X$, which is a composition of $U_i$, a tangle $A \in \B{N}$ and a prime tangle $p_j$. (b) Apply $\tau$ on $A$ and $p_j$ (\Cref{thm:partialTau}). Remember that $e$ passes through only $T$-primes, therefore it is still present in $\tau(A)$. (c) Perform a LZR decomposition on $\tau(A)$ in which $a = i$, $b = N - i - 1$, $a' = j$ and $b' = N - j - 1$. (d) By \Cref{thm:LZRDecompositionTheorem}, $L = I_{i+1} \otimes \sigma$ with $\sigma \in S_{b}$. This implies that we can move $U_i$ to be under $L$. (e) Since $L$ no longer affects edge $e$, we can delete it. (f) Since $R = \omega \otimes I_{N - j}$ with $\omega \in S_j$, then there exists an edge $u = (k,j')$ for some $k < j$. The white circle and the dashed line indicate that we do not know where $k$ lies on the top row. (g) Perform a LZR decomposition on $R$, but since we have that $u$ is connected to $j$, we then have that $L'$ is the identity tangle and can be ignored (remember the special case for \Cref{thm:LZRDecompositionTheorem}). We thus now have $R = Z' \circ R'$. (h) As before, we can move $T_j$ above $R'$. (i) We can now delete $R'$, and this is the simplified case we will prove will always imply a not minimal factorization.}

    \label{fig:simplerCase}
\end{figure}

We will go through all cases for $k$ and prove that this factorization is always not minimal:
\begin{itemize}
    \item if $k > i$, then there are too many crossings (\Cref{fig:non_minimal_tt}):
    \begin{itemize}
        \item the factorization $Z \circ Z' \circ T_j$ has $\ell(Z) + \ell(Z') + 1$ crossings, but edge $e$ and edge $u$ can be redrawn so that they do not cross, thus obtaining the same tangle in $S_N$ but with fewer crossings. This implies that this factorization is not minimal;
    \end{itemize}
    
    \item if $k \leq i$, then the factorization contains $U_i \circ T_{i+1} \circ T_i$ (\Cref{fig:non_minimal_utt}):
    \begin{itemize}
        \item every factor in $Z'$ with index $k < i-1$ can be ignored because 
        \begin{equation*}
        \begin{array}{cc}
             U_i \circ \underbrace{T_{i+1} \circ T_{i+2} \circ \cdots \circ T_{j}}_{Z} \circ \underbrace{\overbrace{T_{k} \circ T_{k+1} \circ \cdots \circ T_{i-2}}^{k < i - 1} \circ T_{i-1} \circ T_i \circ T_{i+1} \circ \cdots \circ T_{j-1}}_{Z'} \circ T_{j} & = \\
            \overbrace{T_{k} \circ T_{k+1} \circ \cdots \circ T_{i-2}}^{k < i - 1} \circ U_i \circ T_{i+1} \circ T_{i+2} \circ \cdots \circ T_{j} \circ T_{i-1} \circ T_{i} \circ T_{i+1} \circ \cdots \circ T_{j-1} \circ T_{j} &  \\ 
        \end{array}
        \end{equation*}
        therefore reducing the factorization to 
        \begin{equation*}
            U_i \circ \underbrace{T_{i+1} \circ T_{i+2} \circ \cdots \circ T_{j}}_{Z} \circ \underbrace{T_{i-1} \circ T_{i} \circ T_{i+1} \circ \cdots \circ T_{j-1}}_{Z'} \circ T_{j}
        \end{equation*}
        By construction, we have that $Z \circ Z' \circ T_j$ will contain an edge $e = (i-1, j'+1)$, which falls into the special case of the LZR decomposition. This implies that we can rewrite the above factorization as
        \begin{equation*}
            U_i \circ Z'' \circ R 
        \end{equation*}
        where $Z'' = T_{i-1} \circ T_i \circ \cdots \circ T_j$ and $R  = \omega \otimes I_{N - j}$, where $\omega \in S_j$. This is clearly not minimal because
        \begin{equation*}
            \begin{array}{lc}
                U_i \circ Z''  &= \\
                U_i \circ T_{i-1} \circ T_i \circ \cdots &=\\
                U_i \circ U_{i-1} \circ \cdots  
            \end{array}
        \end{equation*}

    \end{itemize}
\end{itemize}
in both cases, we have shown that the factorization was not minimal.

\begin{figure}
    \centering
    \newcommand{\scale}{0.18}
    \subfloat[]{\includegraphics[width = \scale\textwidth]{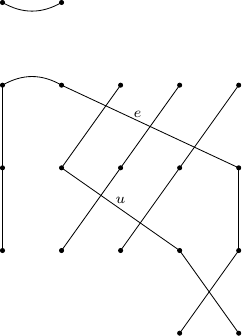}}\hspace{3cm}
    \subfloat[]{\includegraphics[width = \scale\textwidth]{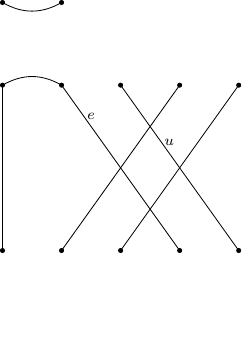}}
    
    \caption{Example for the case in which $k > i$. (a) The factorization $U_i \circ Z \circ Z' \circ T_j$ after LZR decomposition. There are six crossings. (b) The tangle obtained by composing $Z$, $Z'$ and $T_j$. Edges $e$ and $u$ were drawn again so that they do not cross. Note how the resulting tangle has only four crossings, meaning that the previous one was not minimal.}
    
    \label{fig:non_minimal_tt}
\end{figure}

\begin{figure}
    \centering
    \newcommand{\scale}{0.3}
    \subfloat[]{\includegraphics[width = \scale\textwidth]{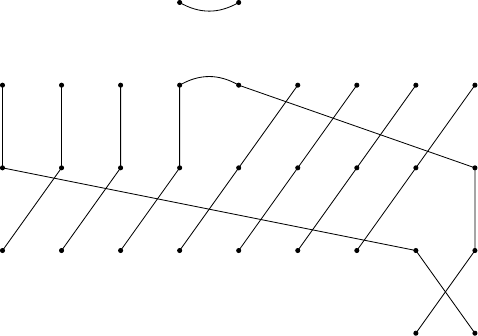}}\hfill
    \subfloat[]{\includegraphics[width = \scale\textwidth]{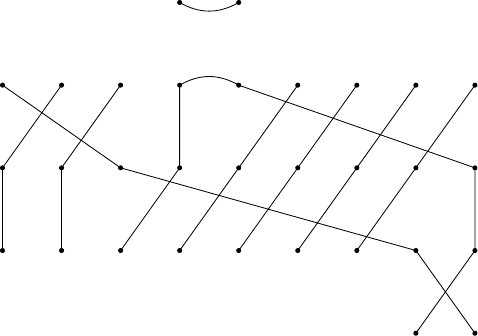}}\hfill
    \subfloat[]{\includegraphics[width = 0.225\textwidth]{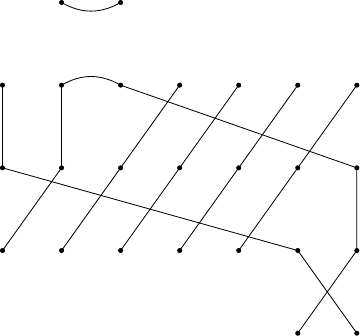}}\hfill\\[3ex]
    \subfloat[]{\includegraphics[width = 0.239\textwidth]{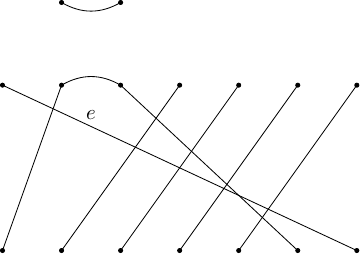}}\hspace{3cm}
    \subfloat[]{\includegraphics[width = 0.27\textwidth]{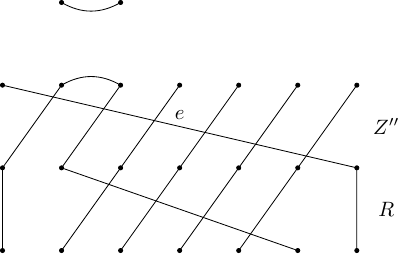}}
    \caption{Example for the case in which $k \leq i$. (a) The factorization $U_i \circ Z \circ Z' \circ T_j$ after LZR decomposition. (b) All $T$-primes with an index less than $i-1$ are removed from their tangle. (c) The $T$-primes are removed. (d) $Z$ and $Z'$ are composed together. By construction, there is an edge $e = (i-1, j'+1)$. (e) Perform a LZR decomposition (special case). By the structure of $Z''$, the resulting factorization is not minimal because it contains $U_i \circ T_{i-1} \circ T_i$. }
    \label{fig:non_minimal_utt}
\end{figure}

There are two more special cases to address. The first one is when edge $e$ is a zero transversal (\Cref{fig:special_case_1}). In this case, the procedure is the same as in \Cref{fig:goingBackExample} but we perform the LZR decomposition only once.

\begin{figure}
    \centering
    
    \newcommand{\scale}{0.2}
    \subfloat[]{\includegraphics[width = \scale\textwidth]{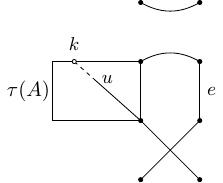}}\hfill
    \subfloat[]{\includegraphics[width = \scale\textwidth]{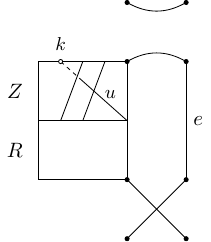}}\hfill
    \subfloat[]{\includegraphics[width = \scale\textwidth]{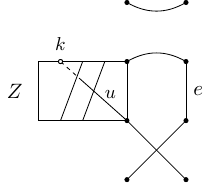}}\hfill
    \subfloat[]{\includegraphics[width = 0.35\textwidth]{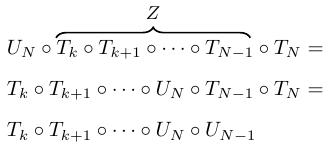}}

    \caption{The special case in which $e$ is a zero transversal. (a) Given a tangle $A \in \B{N}$ and a prime tangle $p_N$, apply $\tau$ to $A$ and $p$. In this case, the index for the $U$-prime is $N$. (b) LZR decomposition of $\tau(A)$. (c) Move $R$ below $T_N$ and remove it. (d) Proof that the factorization is not minimal.}
    \label{fig:special_case_1}
\end{figure}

The last special case is when not only $e$ is a zero transversal, but $A$ has another zero transversal at $(N-1,N'-1)$ (\Cref{fig:special_case_2}). In this case, no simplification step is required because the factorization is trivially not minimal for any prime tangle $p$.

\begin{figure}
    \centering
    \newcommand{\scale}{0.2}

    \subfloat[]{\includegraphics[width = \scale\textwidth]{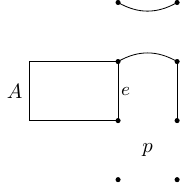}}\hspace{3cm}
    \subfloat[]{\includegraphics[width = \scale\textwidth]{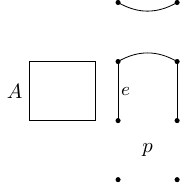}}
    
    \caption{The factorization is not minimal because $U_i \circ p$ is not minimal for any $p$.}
    \label{fig:special_case_2}
\end{figure}

This whole argument proves that $\#U(e) > 0 \implies \#T(e) < |e|$. This is because if it was the case that if $\#U(e) > 0$ but $\#T(e) \geq |e|$, then it would imply that there are more $T$-primes than $|e|$ and therefore, at some point, $e$ must come back, which we just proved is not possible given that $\#U(e) > 0$.
This actually proves that $\#T(e) \geq |e| \implies \#U(e) = 0$ by contraposition.

This also explains why the merge operation is undefined for some edges. If $h = (i,i+1)$ is an upper hook of size one and $e$ is another distinct edge, then $e$ can pass through $U_i$ only if it satisfies one of the conditions presented in \Cref{sec:preliminaries}. If it doesn't, then it must come back to pass through $U_i$, but this would imply that that particular factorization is not minimal, and therefore can be ignored.

\section{Node polarity} \label{sec:nodePolarity}

In this Section we will prove that $\tau(X)$ is unique and computable in polynomial time. To do this, we will introduce the concept of ``node polarity'', a property preserved by $\tau$.

Given a node $x$, we define the \textit{polarity} of $x$ as follows:
\begin{itemize}
    \item if $x$ is connected to a transversal edge $e$, then $x$ is positive (+) if $e$ is positive transversal, negative (-) if it is negative transversal and zero (0) if it is a zero transversal
    \item if $x$ is connected to an upper hook $h$, then $x$ is negative (-) if it is the left node of $h$ and positive (+) if it is its right node. The polarity is reversed for lower hooks.
\end{itemize}

\begin{theorem}\label{thm:tauPreservesPolarity}
    $\tau$ preserves node polarity.
\end{theorem}
\begin{proof}
    For all edges $e$ such that $\#T(e) \geq |e|$ this is trivially true because they are present in both in $X$ and $\tau(X)$, this includes all edges with node polarity 0. Therefore we need to prove that $\tau$ preserves node polarity for all nodes that are connected to edges $e$ such that $\#U(e) > 0$.

    As we can see in the following diagram, after $\tau$ every node of every factor in the factorization will be connected to another one with the same polarity. 
    
    \begin{figure}[H]
        \centering
        \includegraphics[]{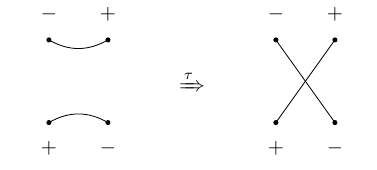}
    \end{figure}

    Since we are assuming that $\#U(e) > 0$, then it implies that $e$ does not come back, and therefore node polarity is preserved.
\end{proof}

Polarity preservation is the first property that $\tau(X)$ must satisfy, but it is not enough because there are many tangles in $S_N$ with the same polarity as $X$. \Cref{lem:perm} and \Cref{thm:polarityDontCross} will state that edges with the same polarity in $\tau(X)$ do not cross, which will imply the uniqueness of $\tau(X)$.

\begin{lemma}\label{lem:perm}
    Let $\sigma$ be a permutation containing an inversion ($\sigma(i)$,$\sigma(j)$), with $i < j$. After swapping $i$ with $j$, the new permutation will have fewer inversions than $\sigma$.
\end{lemma}
\begin{proof}
    For readability's sake, we will define $y = \sigma(i)$ and $x = \sigma(j)$. We also define the notation $\sigma_{I}$ for the set of elements in $\sigma$ having indices in the set $I$.
    
    Let $L$, $C$ and $R$ be three disjoint sets of indices satisfying $i < L,C,R < j$. Let's also assume that
    \begin{equation*}
        \sigma_{L} > y > \sigma_{C} > x > \sigma_{R}
    \end{equation*}
    meaning that the elements between $y$ and $x$ can take any value. We do not need to check the elements outside the rage $[i,j]$ because their number of inversions will stay fixed.

    We can now observe what happens after swapping $y$ with $x$. Moving $y$ will add $|\sigma_{L}|$ inversions because $y$ is smaller than every element in $\sigma_L$, while it will remove $|\sigma_{C}| + |\sigma_{R}|$ inversions because $y$ is bigger than the elements contained in $\sigma_C$ and $\sigma_R$.

    On the other hand, moving $x$ will remove $|\sigma_{L}| + |\sigma_{C}|$ inversions because they contain bigger elements, and it will add $|\sigma_{R}|$ because it contains smaller elements.

    Finally, swapping $y$ and $x$ will remove one inversion because we assumed $y > x$.

    By summing everything together we obtain that the new permutation will have a different amount of inversions, i.e. :
    \begin{align*}
        |\sigma_{L}| - |\sigma_{C}| - |\sigma_{R}| -|\sigma_{L}| - |\sigma_{C}| + |\sigma_{R}| - 1 & =  -2|\sigma_{C}| - 1
    \end{align*}
    We can see that the difference in the number of inversions is always negative, and therefore the new permutation will have fewer inversions.   
\end{proof}

\begin{theorem}\label{thm:polarityDontCross}
    Let $X$ be a tangle with minimal factorization $F$. If the nodes of two edges $e_1, e_2 \in \tau(X)$ have the same polarity and $e_1$ and $e_2$ are not edges of $X$ too, then they do not cross.
\end{theorem}
\begin{proof}
    We know that $\tau$ preserves node polarity and also that since $|\tau(F)|$ is minimal, then it will have the minimal amount of $T$-primes and therefore the minimal amount of crossings in $\tau(X)$. By \Cref{lem:perm} we know that two edges that do not cross add fewer $T$-primes compared to edges that do cross. Therefore $e_1$ and $e_2$ do not cross in $\tau(X)$.
\end{proof}

We can now finally prove that there exists only one $\tau(X)$.

\begin{theorem}\label{thm:uniqueTfication}
    Given a tangle $X$, there exists only one $\tau(X)$ that preserves node polarity and minimizes the number of crossings.
\end{theorem}
\begin{proof}
    For the sake of brevity, we will assume all mentioned nodes are not connected to edges in both $X$ and $\tau(X)$. We will also just focus on nodes with negative polarity, because for positive polarity the argument is basically the same.

    Suppose two upper nodes $x_1 < x_2$ have negative polarity. Assume also that they are the last two upper nodes with negative polarity. Assume now the same for two lower nodes $y_1' < y_2'$.
    
    If we connect $x_1$ with $y_2'$ then it must be the case that $x_2$ has to be connected to $y_1'$, thus introducing a crossing, which is forbidden by \Cref{thm:polarityDontCross}. Therefore we have that $x_2$ must be connected to $y_2'$ in $\tau(X)$. We are now in a situation in which $x_1$ and $y_1'$ are the last nodes that are not connected, and by the previous argument, they must form an edge in $\tau(X)$.

    We can repeat this argument until there are no more edges to connect. Since at each step there was only one possible connection to make, it implies that $\tau(X)$ is unique.
    
\end{proof}

\begin{corollary}
    For all minimal factorizations $F$ and $F'$ for $X$, the tangles corresponding to $\tau(F)$ and $\tau(F')$ are equal.
\end{corollary}

The proof for \Cref{thm:uniqueTfication} gives also an idea of how $\tau$ could be computed. We first start with calculating the node polarity for each node in $X$ from left to right, but every time we find a node with a certain polarity, say ``+'', we will label that node with $+j$, where $j$ is a counter that keeps track of how many nodes with polarity ``+'' we have encountered so far. We can then compute $\tau(X)$ by adding every edge $e$ in $X$ such that $\#T(e) \geq |e|$, and then connecting nodes from top to bottom in $\tau(X)$ if and only if they have the same node polarity label in $X$ (\Cref{algo:tauAlgo}). See \Cref{fig:polarity_tau_example} for an example. \Cref{algo:tauAlgo} has therefore quadratic time complexity because we have to calculate $\#T(e)$ for each edge.

Now that we know that $\tau(X)$ is unique, it is not difficult to see that not only we could count the number of crossings of $\tau(X)$ to find the number of prime tangles for $X$, but we can also factorize $\tau(X)$ to find the indices of those primes. 

\begin{theorem} \label{thm:TauIndices}
    Given a tangle $X$ with minimal factorization $F_X$, then there exists a minimal factorization $F_Y$ for $Y = \tau(X)$ such that $F_Y$ has $T$-primes with the same indices in the same order as in $F_X$.
\end{theorem}
\begin{proof}
    Since $\tau$ does not change the indices, then this is a direct implication of the fact that $\tau(X)$ is the tangle we obtain by composing the factors in $\tau(F_X)$. This factorization can be found by using \Cref{algo:bubbleSort}.
\end{proof}

\begin{figure}
    \centering
    \subfloat[]{\includegraphics[]{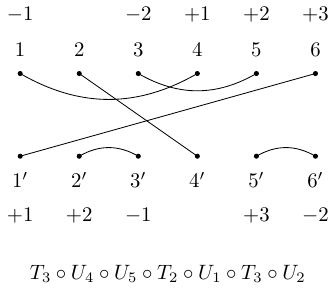}}\hspace{2cm}
    \subfloat[]{\includegraphics[]{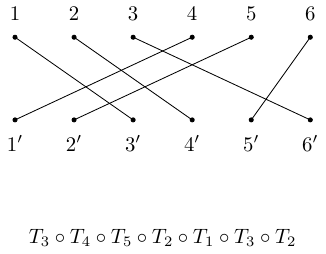}}
    \caption{(a) A tangle $X$ for which we have computed the node polarities along with one of its minimal factorizations. Note that the edge $(2,4')$ is not considered because $\#T(2,4') \geq |(2,4')|$ and therefore passes through only $T$-primes and it is not affected by $\tau$, therefore it is present in both $X$ and $\tau(X)$. (b) The tangle corresponding to $\tau(X)$ and one of it minimal factorizations. To compute it, add the edge $(2,4')$ and connect from top to bottom the nodes that in $X$ have the same node polarity label. Note how the indices for both factorizations coincide.}
    \label{fig:polarity_tau_example}
\end{figure}

\begin{algorithm}
    \caption{Compute $\tau(X)$}
    \begin{algorithmic}
    \Require{$X \in \B{N}$}
        \Function{$\tau$}{$X$}
            \State Let $\rho(i) \in \{+,-\}$ be the polarity for node $i$ in $X$
            \State {$S \gets $ all nodes from $1$ to $N$ connected to edge $e$ s.t. $\#T(e) < |e|$}
            \State {$S' \gets $ all nodes from $1'$ to $N'$ connected to edge $e$ s.t. $\#T(e) < |e|$}
            \State {For $i \in S$, label with ``$\rho(i)j$'' the $j$th node with polarity $\rho(i)$}
            \State {For $i' \in S'$, label with ``$\rho(i')j$'' the $j$th node with polarity $\rho(i')$}
            \State {$Z \gets $ empty tangle in $\B{N}$}
            \State {Add to $Z$ all edges in $X$ s.t. $\#T(e) \geq |e|$}
            \For {nodes $x,y'$ s.t. $x$ and $y'$ have the same node polarity label in $X$}
                \State {Add to $Z$ the edge $(x,y')$}
            \EndFor
            \State \textbf{return} $Z$
        \EndFunction
    \end{algorithmic}
    \label{algo:tauAlgo}
\end{algorithm}

\section{Length function} \label{sec:length_functions}

We are now ready to define two length functions for the Brauer monoid. We will use different subscripts to differentiate between them. The first one, $\ell_P$, is trivially $\ell_{\tau}(X) = \ell(\tau(X))$, where $\ell(\tau(X))$ is just the number of crossings for $\tau(X)$. This function has quadratic complexity because it has to calculate the node polarity first. But even if node polarity could be computed faster, \Cref{thm:TauIndices} tells us that we can obtain the indices for a factorization of $X$ by factorizing $\tau(X)$ (which is very useful information, as we will see in \Cref{sec:finalAlgo}), and factorizing a tangle in $S_N$ cannot be done faster than $\BigO{N^2}$ (see \Cref{sec:app_fact}).

The second length function is a direct corollary of the fact that $\#T(e) \geq |e| \implies \#U(e) = 0$ that we proved in \Cref{sec:passThrough}, because it implies that $\#U(e) > 0 \implies \#T(e) < |e|$ and therefore $\#U(e) > 0 \implies \#T(e) + \#U(e) = |e|$. Let's define the function $\#P(e)$ to be the number of prime tangles the edge $e$ passes through in a factorization $F$.

Since if $\#T(e) \geq |e|$ then $e$ passes through only $T$-primes, and otherwise we have that $\#T(e) + \#U(e) = |e|$, we can define $\#P(e)$ to be:
\begin{equation*}
    \#P(e) = \left\{ \begin{array}{cl}
\#T(e) & : \ \#T(e) \geq |e| \\
|e| & : \ otherwise
\end{array} \right. = max(\#T(e), |e|)
\end{equation*}

Remember that $\#T(e)$ is just the number of crossings edge $e$ has (\Cref{ass:T_are_crossings}). Since now both $\#T(e)$ and $|e|$ are values that are independent of the factorization considered, so is $\#P(e)$. This allows us to use $\#P(e)$ to find a length function for the Brauer monoid by just summing crossings and edge sizes. It is defined as follows:
\begin{equation*}
    \ell_{P}(X) = \frac{1}{2}\sum_{e \in X} \#P(e)
\end{equation*}

Since $\#P(e)$ counts the number of primes $e$ passes through, we have that the sum will count every prime twice. We then have to divide by two to obtain the minimal number of primes for the tangle $X$.

This function still has a quadratic time complexity, but if we assume we already have calculated $\#T(e)$ for all $e \in X$, then it can be computed in linear time (see \Cref{sec:finalAlgo}). We will use this trick in the following Section to bring down the time complexity from $\BigO{N^5}$ to $\BigO{N^4}$.

\section{Factorization algorithm} \label{sec:finalAlgo}

As we have seen in the previous Section, every length function has to be computed in quadratic time, therefore by \Cref{cor:naiveAlgoTimeComplexity} we have that \Cref{algo:naiveFactorization} runs in $\BigO{N^5}$. It turns out however that we can do better. Instead of calculating $\#T(e)$ every time we have to compute $\ell_{P}$, we can store the number of crossings for each edge at the beginning, and then we can just update them every time we merge a lower hook or compose the tangle with a $T$-prime. In this way, updating will have linear time complexity and therefore $\ell_{P}$ will be linear, which brings the overall time complexity to $\BigO{N^4}$.

We update $\#T(e)$ in different ways depending on if we composed with a $T$-prime or merged two edges. In the first case, composing with $T_i$ will modify just two edges, i.e. the ones connected to $i$ and $i+1$. In this case, we just decrease $\#T$ by one for each of them.

In the second case, let's say we merged the lower hook $h$ of size one with the edge $e$. Since we are going to remove them from $X$, we need to decrease $\#T$ for all edges that cross with $e$ (no edge will cross with $h$). Then, after merging $h$ with $e$, two new edges will be in $X$, let's call them $e_1$ and $e_2$. At this point we iterate again through all edges of $X$, and if another edge $d$ crosses with $e_1$, then we increase $\#T(d)$ and $\#T(e_1)$, we then do the same for $e_2$.

In both cases updating $\#T$ takes at most linear time, therefore $\ell_{P}$ is computable in linear time and we have a $\BigO{N^4}$ factorization algorithm, we just have to compute $\#T$ once at the beginning (\Cref{algo:n4_fact}). See \Cref{fig:factorization_example} for an example. A Python implementation can be found at \url{https://github.com/DanieleMarchei/BrauerMonoidFactorization}.

\begin{algorithm}
    \caption{$\mathcal{O}(N^4)$ factorization}
    \begin{algorithmic}
    \Require{$X \in \B{N}$}
        \Function{Factorize $\B{N}$}{$X$}
        \State {Calculate $\#T(e)$ for all $e \in X$}
        \State {$I \gets $ factorization indices of $\tau(X)$ } \Comment{\Cref{algo:tauAlgo} and \Cref{algo:bubbleSort}}
        \State {$F \gets $ empty list}
        \For{$i \in I$}
            \State {$h \gets (i,i+1)$}
            \If {$h \in X$}
                \State {$l \gets \ell_{P}(X)$}
                \For {$e \in X : e \neq h, \#T(e) < |e|$}
                    \State {$X' \gets X$}
                    \For {$d \in X' : d \neq h,e$}
                        \If {$d$ crosses $e$}
                            \State {$\#T(d) \gets \#T(d) - 1$}
                        \EndIf
                    \EndFor
                    \State {Merge $e$ and $h$ in $X'$, creating edges $e_1$ and $e_2$}
                    \For {$d \in X' : d \neq e_1$}
                        \If {$d$ crosses $e_1$}
                            \State {$\#T(e_1) \gets \#T(e_1) + 1$}
                            \State {$\#T(d) \gets \#T(d) + 1$}
                        \EndIf
                    \EndFor
                    \For {$d \in X' : d \neq e_2$}
                        \If {$d$ crosses $e_2$}
                            \State {$\#T(e_2) \gets \#T(e_2) + 1$}
                            \State {$\#T(d) \gets \#T(d) + 1$}
                        \EndIf
                    \EndFor
                    \If {$\ell_{P}(X') = l-1$}
                        \State {Append $U_i$ to $F$}
                        \State {$X \gets X'$}      
                        \State \textbf{break}
                    \EndIf
                \EndFor
            \Else
                \State{$X \gets T_i \circ X$}
                \State {Append $T_i$ to $F$}
                \State {Decrease the number of crossings for the two edges connected at $i$ and $i+1$}
            \EndIf
        \EndFor
        \State {\textbf{return} $F$}
        \EndFunction
    \end{algorithmic}
    \label{algo:n4_fact}
\end{algorithm}

\Cref{algo:n4_fact} could be altered to output a factorization that minimizes the number of $T$-primes without affecting the time complexity by just keeping track of the tangle with the least amount of crossings when merging $e$ and $h$. However, the nested for-loops that search for the edge to merge $h$ with are still the bottleneck of the algorithm. One way to bring down the complexity to $\BigO{N^3}$ could be finding a constant time decision algorithm that determines if a particular edge passes through $U_i$. In this way, finding the edges that can be merged would take linear time and hence reach $\BigO{N^3}$. We were not able to find such an algorithm, so we leave it as a future research direction.

\begin{figure}
    \centering
    \includegraphics[width = \textwidth]{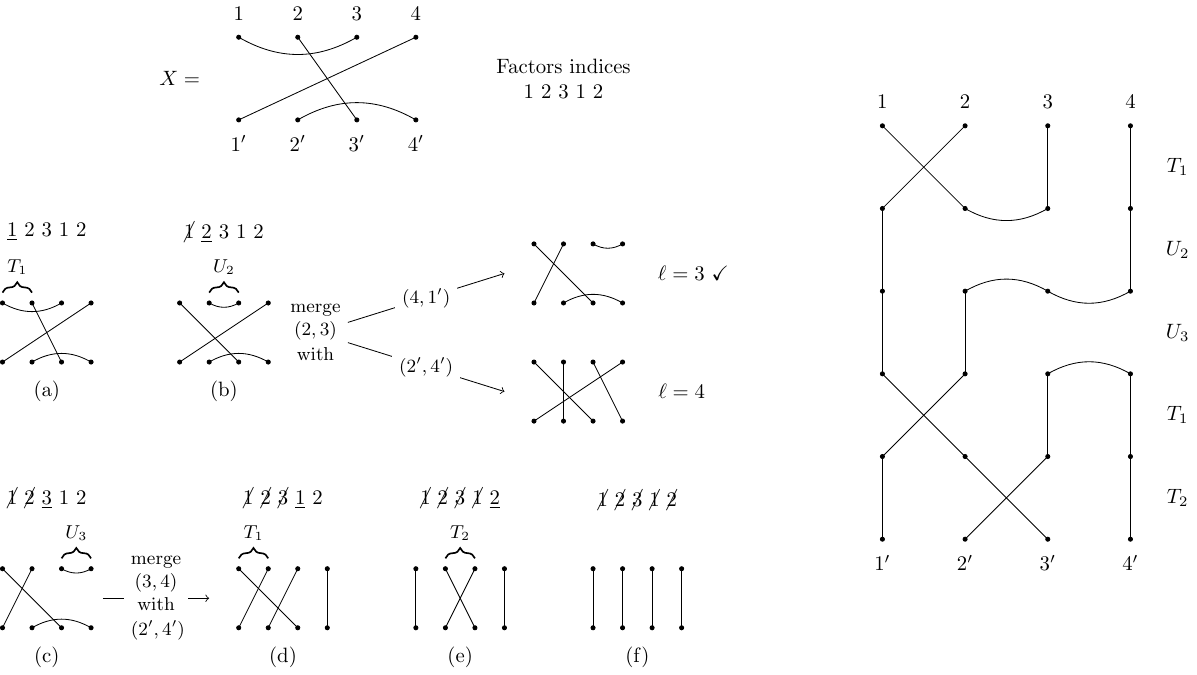}
    \caption{High level illustration for how \Cref{algo:n4_fact} works. On the left, we have a tangle $X$ from which we have extracted its factorization indices. (a) The first factor index is 1, therefore we look at upper nodes 1 and 2 of our tangle and see that they are not connected.  We record $T_1$ and compute $T_1 \circ X$ for the second step. (b) The second index is 2 and the upper nodes 2 and 3 are connected, therefore we record $U_2$. For the next step, we have to decide which edge we have to merge $(2,3)$ with. We have two options: $(4,1')$ and $(2',4')$ (edge $(1,3')$ satisfies $\#T(e) \geq |e|$ so it does not pass through a $U$-prime). If we merge it with $(4,1')$ we obtain a tangle with only three factors, while the other will have four, therefore we select the first one for the next step. (c) We are in the same situation of step (b) but now we can only merge $(3,4)$ with $(2',4')$ and we record $U_3$. (d) Same situation of step (a), we record $T_1$. (e) The factor index is 2 and the upper nodes 2 and 3 are not connected, therefore we record $T_2$. (f) We have reached the identity tangle so we stop the algorithm and output the factorization $T_1 \circ U_2 \circ U_3 \circ T_1 \circ T_2$, which is minimal. On the right we have drawn the factors we have found. It is easy to check that indeed they compose the original tangle $X$. }
    \label{fig:factorization_example}
\end{figure}

\section{Discussion} \label{sec:conclusions}

The Brauer monoid can be factorized in polynomial time, specifically, with a time complexity of $\BigO{N^4}$. To our knowledge, this is the first polynomial time algorithm proposed to solve this problem. We are not sure if it has an optimal running time, since there might be some room for improvements we leave as further research directions. In parallel, we also found two length functions that can be computed in quadratic time, one of which can be computed in linear time assuming the crossing numbers are known. 

Some proofs for this paper rely on two assumptions we were not able to prove nor find in the literature, but their validity has been empirically checked (see \Cref{sec:testAssumptions}). We leave their proof as another future research problem.

The factorization problem could also lead to some interesting combinatorial problems. For example, what is the maximum amount of edges we can merge for a tangle $X$ with a hook $h$ of size one, such that $\ell(X) = \ell(X') + 1$, where $X'$ is the tangle obtained after the merge? This is important to ask because, as we already discussed, finding the right edge to be merged is the bottleneck of \Cref{algo:n4_fact}. By means of enumeration, we obtained \Cref{tab:maxMerges}, and it seems the case that the above questions is answered by $\lfloor\frac{N}{2}\rfloor$, while the number of tangles that have that number of merges is much more difficult to count. For example, the even entries match with the A132911\footnote{\url{https://oeis.org/A132911}} sequence of the OEIS~\cite{sloane2024line}, here we call it $C(k)$:
\begin{equation*}
    C(k) = (k+1)\frac{(2k)!}{2^k}
\end{equation*}
where $k$ starts from zero. To obtain an exact match with the even entries of \Cref{tab:maxMerges}, we will call them $B(2k)$, we have to modify it as follows:
\begin{equation*}
    B(2k) = C(k-1) = k!|\B{k-1}|
\end{equation*}
where $k$ starts from one. 

\begin{table}[H]
    \centering
    \begin{tabular}[c]{|l|ll}
        \hline
        $N$ & max amount of merges & n. tangles $B(N)$ \\
        \hline
        2 & 1 & 1 \\
        3 & 1 & 6 \\
        4 & 2 & 2 \\
        5 & 2 & 46 \\
        6 & 3 & 18 \\
        7 & 3 & 900 \\
        8 & 4 & 360 \\
        9 & 4 & 31320 \\
        10& 5 & 12600 \\ 
    \end{tabular}
    \caption[The maximum amount of possible merges in $\B{N}$.]{The maximum amount of possible merges in $\B{N}$ and the number of tangles $X$ with a hook $h$ of size one that can be merged with other edges such that $\ell(X) = \ell(X') + 1$, where $X'$ is the tangle obtained after the merge.}
    \label{tab:maxMerges}
\end{table}

Another question could be: how many tangles have length $k$ in $\B{N}$? Using the results presented, we enumerated all tangles up to $\B{10}$ and obtained \Cref{tab:NumTangleWithFactors}, let's call it $T(N,k)$. Some clear pattern emerge, for example $T(N,1) = 2(N-1)$ (as expected),  or $T(N,1) = T(N-2, \frac{N(N-1)}{2}-1)$ for $N \geq 5$, but we were unable to find a general formula. 

We don't have a proof for any of the above statements, nor we have a candidate formula for the odd entries. We leave these questions open as a further research direction.

\begin{table}[H]
    \small
    \begin{tabular}{|l|llllllllll}
    \hline
    \theadfont\diagbox[width = 3em]{$k$}{$N$} & 1 & 2 & 3 & 4 & 5 & 6 & 7 & 8 & 9 & 10 \\
    \hline
    0 & 1 & 1 & 1 & 1 & 1 & 1 & 1 & 1 & 1 & 1 \\
    1 &  & 2 & 4 & 6 & 8 & 10 & 12 & 14 & 16 & 18 \\
    2 &  &  & 8 & 20 & 36 & 56 & 80 & 108 & 140 & 176 \\
    3 &  &  & 2 & 36 & 102 & 208 & 362 & 572 & 846 & 1192 \\
    4 &  &  &  & 30 & 196 & 562 & 1224 & 2294 & 3900 & 6186 \\
    5 &  &  &  & 10 & 228 & 1110 & 3192 & 7266 & 14380 & 25870 \\
    6 &  &  &  & 2 & 212 & 1650 & 6620 & 18746 & 43764 & 90034 \\
    7 &  &  &  &  & 106 & 1966 & 11090 & 40166 & 112250 & 266462 \\
    8 &  &  &  &  & 42 & 1914 & 15890 & 73278 & 247494 & 682770 \\
    9 &  &  &  &  & 12 & 1440 & 19442 & 116996 & 477830 & 1538840 \\
    10 &  &  &  &  & 2 & 830 & 20910 & 166400 & 825422 & 3100160 \\
    11 &  &  &  &  &  & 414 & 18798 & 212250 & 1291638 & 5667090 \\
    12 &  &  &  &  &  & 162 & 15402 & 244730 & 1853554 & 9514646 \\
    13 &  &  &  &  &  & 56 & 10174 & 255188 & 2448214 & 14804426 \\
    14 &  &  &  &  &  & 14 & 6154 & 240828 & 3003652 & 21502064 \\
    15 &  &  &  &  &  & 2 & 3282 & 207968 & 3411904 & 29298972 \\
    16 &  &  &  &  &  &  & 1530 & 161844 & 3627806 & 37604566 \\
    17 &  &  &  &  &  &  & 648 & 113490 & 3585522 & 45596280 \\
    18 &  &  &  &  &  &  & 234 & 73978 & 3325568 & 52372154 \\
    19 &  &  &  &  &  &  & 72 & 44336 & 2856302 & 57069858 \\
    20 &  &  &  &  &  &  & 16 & 24354 & 2325126 & 59057576 \\
    21 &  &  &  &  &  &  & 2 & 12462 & 1741684 & 58153920 \\
    22 &  &  &  &  &  &  &  & 5848 & 1238988 & 54397782 \\
    23 &  &  &  &  &  &  &  & 2502 & 830378 & 48420890 \\
    24 &  &  &  &  &  &  &  & 972 & 523782 & 41150508 \\
    25 &  &  &  &  &  &  &  & 324 & 312886 & 33243338 \\
    26 &  &  &  &  &  &  &  & 90 & 176806 & 25585214 \\
    27 &  &  &  &  &  &  &  & 18 & 94362 & 18883774 \\
    28 &  &  &  &  &  &  &  & 2 & 47280 & 13337554 \\
    29 &  &  &  &  &  &  &  &  & 22294 & 9028454 \\
    30 &  &  &  &  &  &  &  &  & 9756 & 5856940 \\
    31 &  &  &  &  &  &  &  &  & 3908 & 3653772 \\
    32 &  &  &  &  &  &  &  &  & 1406 & 2186074 \\
    33 &  &  &  &  &  &  &  &  & 434 & 1253770 \\
    34 &  &  &  &  &  &  &  &  & 110 & 688446 \\
    35 &  &  &  &  &  &  &  &  & 20 & 361372 \\
    36 &  &  &  &  &  &  &  &  & 2 & 180488 \\
    37 &  &  &  &  &  &  &  &  &  & 85298 \\
    38 &  &  &  &  &  &  &  &  &  & 37930 \\
    39 &  &  &  &  &  &  &  &  &  & 15636 \\
    40 &  &  &  &  &  &  &  &  &  & 5880 \\
    41 &  &  &  &  &  &  &  &  &  & 1972 \\
    42 &  &  &  &  &  &  &  &  &  & 566 \\
    43 &  &  &  &  &  &  &  &  &  & 132 \\
    44 &  &  &  &  &  &  &  &  &  & 22 \\
    45 &  &  &  &  &  &  &  &  &  & 2
    \end{tabular}
    \caption{The number of tangles in $\B{N}$ with length $k$.}
    \label{tab:NumTangleWithFactors}
    \end{table}

\section*{Acknowledgements}
We would like to thank James East, Matthias Fresacher, Alfilgen Sebandal and Azeef Parayil Ajmal for their invaluable discussions and suggestions. We would also like to thank James Mitchell for helping us with the time complexity for Algorithm 13 of ``Computing finite semigroups''.

\appendix
\section{Factorization algorithms for submonoids of $\B{N}$} \label{sec:app_fact}

\subsection{Factorization of $S_N$}

Every element in $S_N$ can be uniquely described as a permutation of the string $1,2,\cdots,N$. A simple isomorphism from a string permutation $s_1, s_2, \cdots, s_N$ and a tangle in $S_N$ is to connect the upper node $i$ to the lower node $s_i'$. The identity permutation string is therefore the one in which every element is in ascending order. In this context, factorizing a tangle in $S_N$ is the same as finding the shortest sequence of adjacent transpositions (i.e. $s_{i}, s_{i+1} \rightarrow s_{i+1}, s_i$) such that the original string is reduced to the identity string. In other words, we have to sort the string. This is a well-known problem in Computer Science and there are numerous fast algorithms for solving it. However, due to the constraint of using only adjacent transpositions, we are bound to a quadratic time complexity since the longest minimal factorization in $S_N$ has size $\frac{N(N-1)}{2}$. One such algorithm is the \textsc{BubbleSort} \cite{knuth1998theart}, which specifically sorts strings using adjacent transpositions, yielding the shortest possible sequence (see \Cref{algo:bubbleSort}).

\begin{algorithm}
    \caption{Factorization algorithm for any tangle in $S_N$.}
    \begin{algorithmic}
    \Require $X \in S_N$
        \Function{Factorize $S_N$}{$X$}
            \State {$s \gets $ string permutation for $X$}
            \State {$F \gets$ empty list}
            \For {$j \in [N \dots 1]$}
                \For {$i \in [1 \dots j-1]$}
                    \If {$s_i > s_{i+1}$}
                        \State {Swap $s_i$ and $s_{i+1}$}
                        \State {Append $T_i$ to $F$}
                    \EndIf
                \EndFor
            \EndFor
            \State \textbf{return} $F$
        \EndFunction
    \end{algorithmic}
    \label{algo:bubbleSort}
\end{algorithm}

\subsection{Factorization of $\TL_N$}
The Temperley-Lieb monoid $\TL_N$ \cite{TL-original} is a submonoid of $\B{N}$ in which only $U$-primes are taken as generators. Informally speaking, $\TL_N$ contains all tangles with no crossings. A $\Omeg{N^2}$ factorization algorithm was first proposed by Ernst et. al ~\cite{ernst2016factorization} and in this Section we will give a surface-level explanation of how it works (see \Cref{fig:tlAlgo}).

\begin{figure}
    \centering
    \subfloat[]{\includegraphics[width=.33\textwidth]{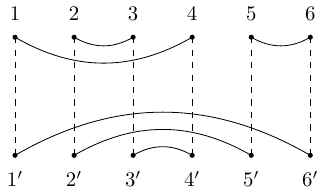}}\hfill
    \subfloat[]{\includegraphics[width=.33\textwidth]{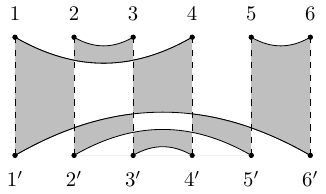}}\hfill
    \subfloat[]{\includegraphics[width=.33\textwidth]{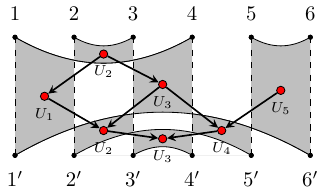}}
    \caption{(a) A tangle in $\TL_N$. Drawing $N$ imaginary edges $(i,i')$ we see that the tangle is now divided into $N-1$ columns and each column is divided into regions, delimited by edges. (b) Select the regions with odd depth. (c) If two regions are diagonally adjacent, connect the top one with the bottom one, thus obtaining a Directed Acyclic Graph. Label each node as $U_i$, where $i$ is the index of the columns it is in. If we read this DAG from top to bottom and from left to right we obtain the minimal factorization $U_2 \circ U_5 \circ U_1 \circ U_3 \circ U_2 \circ U_4 \circ U_3$.}
    \label{fig:tlAlgo}
\end{figure}

Given a tangle $X$ in $\TL_N$, divide it into $N-1$ columns by adding a vertical line for each pair $(i,i')$. Each column now will be further divided into different regions delimited by the edges it contains. Each region has a depth value indicated by how many other regions there are above it. We will call regions with even depth 0-regions and regions with odd depth 1-regions. Two regions, $R_1$ and $R_2$, in the same column are vertically adjacent if $|depth(R_1) - depth(R_2)| = 1$. Given two regions $R_1$ and $R_2$ in adjacent columns and two points $p_1 \in R_1$ and $p_2 \in R_2$, if we can draw a straight line between them without crossing an edge in $X$, then we say that $R_1$ and $R_2$ are horizontally adjacent. If there is a region $R'$ vertically adjacent to $R_1$ and horizontally adjacent to $R_2$, then $R_1$ and $R_2$ are diagonally adjacent. We will $R_1 \rightarrow R_2$ in the special case where $R'$ is below $R_1$.

Let $\mathfrak{R}$ be the set of all 1-regions. From here we construct a Directed Acyclic Graph (DAG) $G$ such that every region $R \in \mathfrak{R}$ is a vertex of $G$ and given two regions $R_1$ and $R_2$, if $R_1 \rightarrow R_2$, then $(R_1, R_2)$ is an edge of $G$. If a vertex does not have incoming edges, then we call it a root of $G$. Finally, to obtain the factorization for $X$ we traverse each vertex from \textit{top to bottom} and from \textit{left to right}, i.e. we list all roots $r_i, r_j, \dots, r_k$ of $G$, store $U_i \circ U_j \circ \cdots \circ U_k$ and delete the roots, now other nodes will not have incoming edges, list them as the new roots $r_i', r_j', \dots, r_k'$, and so on until $G$ is empty (\Cref{algo:TLAlgorithm}).

\begin{algorithm}
\caption{Factorizes tangle in $\TL_N$}
\begin{algorithmic}
\Procedure{Factorize $\TL$}{$X$}
    \State{Divide $X$ in columns}
    \State{$R \gets $ regions of $X$}
    \State{$\mathfrak{R} \gets $ 1-regions of $R$}
    \State{$G \gets $ DAG from $\mathfrak{R}$ s.t. if $R_1 \rightarrow R_2$, then $(R_1, R_2) \in G$}
    \State{$F \gets $ empty list}
    \While{$G \neq \emptyset$}
        \State {$S \gets $ roots of $G$}
        \For{$s \in S$}
            \State {$i \gets$ column where $s$ lies in}
            \State {Append $U_i$ to $F$}
        \EndFor
        \State {Remove nodes in $S$ from $G$}
    \EndWhile
    \State \textbf{return} $F$
\EndProcedure
\end{algorithmic}
\label{algo:TLAlgorithm}
\end{algorithm}

The time complexity of this algorithm is bounded below by $\Omeg{N^2}$ because the number of 1-regions is the same as the number of nodes in $G$, which is equal to the length of the factorization for the input tangle and, by \Cref{cor:longestMinimalFact}, we know it is quadratic. Assuming that retrieving the roots of $G$ can be done in constant time, the while loop basically enumerates the nodes of $G$, therefore it does not influence the time complexity. If we assume that the number of total regions is still quadratic and they could be enumerated in quadratic time as well, we have that \Cref{algo:TLAlgorithm} has a time complexity of $\BigO{N^2}$.

\section{Testing the Assumptions} \label{sec:testAssumptions}

Here we present how methodology for empirically testing \Cref{ass:reductionInBN} and \Cref{ass:T_are_crossings}. The interested reader can find the Python code in the GitHub repository \url{https://github.com/DanieleMarchei/BrauerMonoidFactorization}. 

\subsection{Preparation}
To test the two assumptions we need a database of minimal factorizations. We cannot use \Cref{algo:n4_fact}, as it would be circular reasoning. This is because its correctness relies on them being true (it is used in \Cref{sec:passThrough} to partially apply $\tau$ and in \Cref{sec:length_functions} to argue that $\ell_P$ is independent of the factorization). Thus we decided to create this database by exploring the right Cayley Graph of $\B{N}$ using a Breadth First Search (BFS)\footnote{According to \cite{erickson2023algorithms}, the BFS algorithm was first proposed in \cite{zuse1946Plankalkul} and not in \cite{moore1959shortest}, as it is often attributed.} starting from the identity tangle. The BFS algorithm has the nice property of always returning the shortest path in a graph (which will be a minimal factorization in our case) and by instructing it to first explore the edges labelled with a $U$-prime, and then all edges labelled with a $T$-prime, every time we reach an unexplored tangle we know we have obtained a minimal factorization that uses the least amount of $T$-primes. We ran this procedure up to $\B{8}$.

\subsection{Assumption 1} \label{sec:testAssumption1}

\Cref{ass:reductionInBN} stated that every factorization in the Brauer monoid can be reduced to a minimal one by a sequence of ``delete'', ``braid'' and ``swap'' rules. In other words, we do not need to increase the factorization length in order to find a shorter one.

To test this empirically, we implemented the axioms for $\B{N}$ (\Cref{tab:axioms}) as a Term Rewriting System (TRS) using the Maude System \cite{clavel2022maude}. The approach is similar to what described in \cite{marchei2022rna}, but to have this paper self-contained, we will give a high-level illustration of the procedure.

The TRS accepts in input a factorization and iteratively tries to apply as many ``delete'' rules it can. Then, it tries to apply ``braid'' and ``swap'' rules non-deterministically until a new delete rule is applicable, and at this point the TRS starts again with the newly shortened factorization. This reduction procedure stops when no delete rule is applicable after all move rules have been applied.

The actual test for the assumption is performed as follows:
\begin{enumerate}
    \item generate a random factorization $F$ with length $|F| \in [2, s\frac{N(N-1)}{2}]$ and compute the corresponding tangle $X \in \B{N}$;
    \begin{itemize}
        \item if the factorization for $X$ stored in the database has the same length of $F$, then repeat step 1. We only want to test non-minimal factorizations 
    \end{itemize}
    \item get the minimal factorization $F^*$ for $X$ in the database;
    \item reduce $F$ using the TRS, obtaining $\hat{F}$;
    \item if $|\hat{F}| \neq |F^*|$ we have found a counterexample, otherwise repeat from step 1.
\end{enumerate}
where $s$ is a scale factor that parametrizes the maximum length possible that can be generated. To avoid spending too much time on this procedure, we keep track of the tangles generated, effectively testing the assumption once for each tangle. Since the probability of generating a tangle we haven't already generated decreases each time we find a new one, we implemented a ``patience'' counter that is decreased each time we do not generate a new tangle. The procedure stops when the patience reaches to zero or all tangles are tested. In our test we set $s = 2$ and the patience counter to $2000000$.
Because this procedure takes a long time to terminate, we ran it up to $N = 6$ with the results shown in \Cref{tab:ass1test}. We never found a counterexample.
\begin{table}[H]
    \centering
    \begin{tabular}{c|c|c|c|c|c}
         $N$ & 2 & 3 & 4 & 5 & 6 \\
         \hline
         $|\B{N}|$ & 3 & 15 & 105 & 945 & 10395 \\
         \hline
         n. tangles tested & 3 & 15 & 105 & 942 & 10043 \\
    \end{tabular}
    \caption{Test table for \Cref{ass:reductionInBN}. We ran the test up to $\B{6}$ and, among all tangles tested, we never found a counterexample.}
    \label{tab:ass1test}
\end{table}

\subsection{Assumption 2} \label{sec:testAssumption2}

\Cref{ass:T_are_crossings} stated that if a tangle $X$ has $k$ crossings, then there exists a minimal factorization $F$ with exactly $k$ $T$-primes and no other factorization with fewer $T$-primes exists. 

This assumption is easier and faster to check:
\begin{enumerate}
    \item pick a tangle $X$ with minimal factorization $F$ from the database;
    \begin{itemize}
        \item by construction, $F$ will have the minimal amount of $T$-primes.
    \end{itemize}
    \item count the number of crossings $c$ in $X$;
    \item if $c \neq \#T(F)$ then we have found a counterexample, otherwise repeat from step 1.
\end{enumerate}

We ran this procedure up to $N = 8$ and never found a counterexample.

\end{document}